\documentclass[12pt]{amsart}
\usepackage{amssymb,verbatim,enumerate,ifthen}
\usepackage[mathscr]{eucal}
\usepackage[utf8]{inputenc}
\usepackage[T1]{fontenc}
\makeatletter
\@namedef{subjclassname@2010}{%
  \textup{2010} Mathematics Subject Classification}
\makeatother

\newtheorem{thm}{Theorem}
\newtheorem{cor}{Corollary}
\newtheorem{prop}{Proposition}
\newtheorem{lem}{Lemma}

\newtheorem{rem}{Remark}

\newtheorem*{xrem}{Remark}

\numberwithin{equation}{section}

\newcommand\nolabel[1]{\nonumber}

\newcommand\A[1]{P_{#1}}
\newcommand\R{\mathbb{R}}

\newcommand\N{\mathbb{N}}

\newcommand{\sr}[1]{A^{[#1]}}

\newcommand{\srpot}[1]{\mathcal{P}_{#1}}
\newcommand{\norma}[1]{\left\| #1 \right\| }
\newcommand{\abs}[1]{\left| #1 \right| }

\newcommand{\Sm}{\mathcal{S}}
\newcommand{\srE}[1]{\mathcal{E}_{#1}}

\DeclareMathOperator{\dom}{dom}

\newcommand{\constantCzero}{1.24886}
\newcommand{\constantyzero}{0.00007314}
\newcommand{\constantyone}{0.001608}

\subjclass[2010]{26E60, 26D15, 26D07} 
\keywords{quasi-arithmetic means, metric, Arrow-Pratt index, lower boundaries, risk aversion, distance among means}
\title{Lower estimation of the difference among quasi-arithmetic means}

\date{\today}
\begin{document}

%%%%% To ease editing, for IMPAN journals add:

\baselineskip=17pt

%%%%%%%%%%%

%% In the running head, replace first names by initials 
%% and give an abbreviation of the title.

\numberwithin{equation}{section}
\def\eq#1{{\rm(\ref{#1})}}
\def\Eq#1#2{\ifthenelse{\equal{#1}{*}}
  {\begin{equation*}\begin{aligned}[]#2\end{aligned}\end{equation*}}
  {\begin{equation}\begin{aligned}[]\label{#1}#2\end{aligned}\end{equation}}}

\author[P. Pasteczka]{Pawe\l{} Pasteczka}
\address{Institute of Mathematics \\ Pedagogical University of Cracow \\ Podchor\k{a}\.zych str. 2, 30-084 Krak\'ow, Poland}
\email{ppasteczka@up.krakow.edu.pl}

\maketitle

\begin{abstract}
Quasi-arithmetic means are defined for every continuous, strictly monotone function $f \colon U \rightarrow \R$, ($U$ -- an interval). For an $n$-tuple $a \in U^n$ with corresponding vector of weights $w=(w_1,\dots,w_n)$ ($w_i>0$, $\sum w_i=1$) it equals $f^{-1}\left( \sum_{i=1}^{n} w_i f(a_i)\right)$.

In 1960s Cargo and Shisha defined a metric in a family of quasi-arithmetic means defined on a common interval as the maximal possible difference between these means taken over all admissible vectors with corresponding weights. 

During the years 2013--16 we proved that, having two quasi-arithmetic means, we can majorized distance between them in terms of Arrow-Pratt index $f''/f'$.
In this paper we are going to proof that this operator can be also used to establish certain lower boundaries of this distance.
\end{abstract}

\section{Introduction}
One of the most popular family of means encountered in literature is the family of 
\textit{quasi-arithmetic means}. These means are defined for any continuous, strictly 
monotone function $f \colon U \rightarrow \mathbb{R}$, $U$ -- an interval. 
When $a = (a_1,\dots,a_n)$ is an arbitrary sequence of points in $U$
and $w=(w_1,\dots,w_n)$ is a sequence of corresponding \textit{weights}
($w_i>0$, $\sum w_i=1$), then the mean $\sr{f}(a,w)$ is defined by the equality
\Eq{*}{
\sr{f}(a,w) := f^{-1}\left( \sum_{i=1}^{n} w_i f(a_i)\right).
}
This family of means first only glimpsed in the pioneering paper by Knopp ~\cite{Kno28}; shortly after it was formally introduced in a series of nearly simultaneous papers 
\cite{Def31,Kol30,Nag30} at the beginning of the 1930s. In fact quasi-arithmetic means has been considering as a generalization of well-known family of Power Means ever since their first appearing in \cite{Kno28}. Indeed, upon putting $U=(0,+\infty)$, $p_s(x)= x^s $ for $s \ne 0$, completed by $p_0(x)=\ln x$ we could easily identify power means as a subfamily of quasi-arithmetic means. Let us specify one more classical family of quasi-arithmetic means that will be used in this paper -- let $\{e_s \colon \R \to \R\}_{s \in \R}$ will be given by $e_s(x) = \exp(s \cdot x)$ for $s \ne 0$, completed by $e_0(x):=x$.
Denote by $\srE{s}$ the quasi-arithmetic mean generated by $e_s$, that is $\srE{s}:=\sr{e_s}$  for every $s \in \R$. In the other words for every vector $a=(a_1,\dots,a_n)$ of reals and corresponding weights $w$ we get
\Eq{*}{
\srE{s}(a,w)=\sr{e_s}(a,w)=
\begin{cases}
\tfrac1s \ln \left( \sum_{i=1}^n w_i e^{s\cdot a_i} \right) & \text{, if } s \ne 0, \\
             \sum_{i=1}^n w_ia_i  & \text{, if } s = 0.
 \end{cases}
}
These means are sometimes called $\log\text{-}\exp$ means (cf. \cite[p. 269]{Bul03}).
Family $(\srE{s})_{s \in \R}$ is closely related with the family of Power Means (denoted here as $\srpot{s})$. 
For $n \in \N$, $a \in \R^n$, and corresponding weights $w$ there holds
\Eq{*}{
\srpot{s}\big((e^{a_1},e^{a_2},\dots,e^{a_n}),w\big)=\exp(\srE{s}(a,w)) \text{ for every } s \in \R.
}

In fact this relation is even closer -- $(\srpot{s})_{s \in \R}$ and $(\srE{s})_{s \in \R}$ are two examples of so-called \emph{invariant scales} (cf. \cite{Pas15c} for details).

\vskip2mm

Coming back to a [general] quasi-arithmetic means. This means have been develop ever since their definition. In particular, in 1960s Cargo and Shisha \cite{CarShi69} introduced a metric among them. Namely, if $f$ and $g$ are both continuous, strictly monotone and have the same domain, then one can define a distance 
\Eq{*}{
\rho(\sr{f},\sr{g}):=\sup\{\abs{\sr{f}(a,w)-\sr{g}(a,w)} \colon a\textrm{ and }w \textrm{ admissible}\}.
}
Original definition were established in a different way, nevertheless this wording is equivalent.
Main goal of the present paper is to establish some {\it lower bounds} of the distance $\rho$. In \cite{CarShi69}, there was presented a number of majorizations of $\rho$. In particular few of them concerns lower boundaries
\begin{prop}[\cite{CarShi69}, Lemma~3.1, Theorem~4.2]
\label{prop:CS69}
Let $f,g \colon [0,1] \rightarrow [0,1]$ be two strictly increasing functions satisfying $f(0)=g(0)=0$, and $f(1)=g(1)=1$. Then
\Eq{*}{
\rho(\sr{f},\sr{g}) \ge \norma{f^{-1}-g^{-1}}_{\infty}.
}
Moreover, if $f^{-1}$ has a bounded derivative on $[0,1]$ then 
\Eq{*}{
\rho(\sr{f},\sr{g})\le 2\norma{\left(f^{-1}\right)'}_{\infty} \cdot \norma{f-g}_{\infty}.
}
\end{prop}

\begin{xrem}
 The assumption about the domain could be omitted by using the scaling of functions $f$ and $g$.
\end{xrem}

In our setting we are going to establish lower boundaries in terms of the Arrow-Pratt index (defined both in the abstract and below).

By \cite[Theorem~6.1]{CarShi69} the maximal value of the distance between quasi-arithmetic means is obtained for a vector of length two with some weights (it obviously does not imply that it cannot be obtained for any other vector of entries). More precisely in the definition above we can assume without loss of generality that the vector $a$ has length two. Thus it is reasonable to restrict our consideration to $2$-entry vectors only. In this setting it will be handy to denote 
\Eq{*}{
\sr{f}_\theta(z,x):=f^{-1}(\theta f(z)+(1-\theta)f(x)).
}
Therefore, using the notation introduced just an instant ago, we obtain the equality
\Eq{*}{
\rho(\sr{f},\sr{g})=\sup_{x,\,z\in U \atop \theta \in (0,1)}\abs{\sr{f}_\theta(z,x)-\sr{g}_\theta(z,x)}.
}

This equality boils down the problem of finding the distance between two quasi-arithmetic means to find the maximum value of some function in the space $U^2 \times (0,1)$.

Notice that the problem of finding a lower and upper estimates of the difference is much different. To establish an upper boundary we are looking for the inequality, which is valid for all entries (and weights) -- this is a very common method whenever $\sup$ operator appear.
Contrary to this, when it comes to finding a lower boundary, we need to prove that there exists a vector (with corresponding weights) such that the distance between two quasi-arithmetic means evaluated on it can be bounded from below. Additionally, it natural to look for an estimations which are resistant under affine changing of the functions (cf. Remark~\ref{rem:eq_qa} below).

First possible solution comes from Mikusi\'nski \cite{Mik48} and, independently, \L{}ojasiewicz (cf. \cite[footnote 2]{Mik48}). For a twice differentiable function $f \colon U \rightarrow \R$ ($U$ - an interval) having nowhere vanishing first derivative (hereafter we will denote such a family of functions by $\Sm(U)$ ) they defined an operator $\A{f}:=f''/f'$. In the mathematical economy, the {\it negative\,} of this operator happened to be called the {\it Arrow-Pratt measure of risk aversion}, in dynamical systems it is called {\it nonlinearity} of the function.
Besides, Mikusi\'nski and \L{}ojasiewicz proved that comparability of quasi-arithmetic means might be easily expressed in terms of operator $\A{}$.

\begin{prop}[Basic comparison]\label{prop:basiccompare}
Let $U$ be an interval, $f,\,g \in \Sm(U)$. 
Then the following conditions are equivalent: 
%%%%%%%%%%
\begin{enumerate}[\upshape (i)]
%%%%
\item $\A f > \A g$ on a dense set in $U$\,,
%%%%
\item $(\mathrm{sgn} f') \cdot (f \circ g^{-1})$ is strictly convex\,,
%%%%
\item $\sr{f}(a,w) \ge \sr{g}(a,w)$ for all vectors $a \in U^n$ and weights 
$w$, with both sides equal only when $a$ is 
a constant vector.
\end{enumerate}
\end{prop}
Mikusi\'nski proved the equivalence \upshape{(i)} $\iff$ \upshape{(iii)}, while the  \upshape{(ii)} $\iff$ \upshape{(iii)} part is simply implied by Jensen's inequality.

Second solution uses the three-parameters-operator 
\Eq{*}{
\{ (x,y,z) \in U^3 \colon x \ne z \} \ni (x,y,z) \mapsto \frac{f(x)-f(y)}{f(x)-f(z)} \in \R
}
introduced by P\`ales in \cite{Pal91}. He proved that the pointwise convergence of quasi-arithmetic means can be easily expressed in terms of this operator. There was, however, no results estimating the distance $\rho$ using this mapping. Some approach was given recently in \cite{Pas15d}. More precisely we have the following
\begin{prop}
Let $U$ be an interval, $f,\,g \colon U \rightarrow \R$ be two continuous, strictly monotone functions and $\alpha>0$.  
If 
\Eq{*}
{
\abs{\frac{f(x)-f(y)}{f(x)-f(z)} - \frac{g(x)-g(y)}{g(x)-g(z)} } <C<1 \text{ for all } (x,y,z) \in U^3 \text{ such that } \abs{x-z} \ge \alpha
}

then $\rho(\sr{f},\sr{g})\le \alpha$.
\end{prop}

There is one crucial reason making the operation $\A{}$ and the mapping above very natural to describing quasi-arithmetic means. Namely, Proposition~\ref{prop:basiccompare} has its equal-type counterpart:
\begin{rem}
\label{rem:eq_qa}
Let $U$ be an interval, $f,\,g \colon U \rightarrow \R$ be continuous and strictly monotone functions.
Then the following conditions are equivalent: 
%%%%%%%%%%
\begin{enumerate}[\upshape (i)]
%%%%
\item $\sr{f}(a,w) =  \sr{g}(a,w)$ for all vectors $a \in U^n$ and corresponding weights $w$;
%%%%
\item $f(x)=\alpha g(x)+\beta$ for some $\alpha,\beta \in \R$ and every $x \in U$\,;
%%%%
\item $\frac{f(x)-f(y)}{f(x)-f(z)} =\frac{g(x)-g(y)}{g(x)-g(z)}$ for all $(x,y,z) \in U^3$ such that $x \ne z$.
%%%%
\end{enumerate}

Moreover if $f,\,g \in \Sm(U)$ we have an additional condition

\begin{enumerate}[\upshape (i)]
%%%%
\item[\upshape (iv)] $\A{f}(x)=\A{g}(x)$ for every $x \in U$\,.
\end{enumerate}
\end{rem}

Having this in hand, whenever the final result is stated in terms of operator $\A{}$, we do not have to make any extra assumption involving affine transformations of generating functions -- contrary to Proposition~\ref{prop:CS69}.

In what follows we are going to present a number of results majorizing the difference between two quasi-arithmetic means in terms of operator $\A{}$. Next, we are going to present certain lower bounds of operator $\rho$ in the general setting (section~\ref{sec:main}), and under the stronger assumption (section~\ref{sec:box}).

\subsection{Operator $\A{}$\: in estimating differences among quasi-arithmetic means}
First result majorizing differences among quasi-arithmetic means using Arrow-Pratt index was established in \cite{Pas13}. We do not recall this result here, because it was strengthened in \cite{Pas15d} in terms of a special norm $\norma{\cdot}_\ast$ defined by 
\Eq{*}{
\norma{f}_\ast=\sup_{x,y\in \dom f} \abs{\int_x^y f(t) dt}.
}
For a subinterval $U \subset \dom f$ we will also define $\norma{f}_{\ast,U}:=\norma{f \vert_{U}}_{\ast}$.
\begin{rem}
\label{rem:star_norm}
The family of measurable functions defined on fixed interval and having finite $\ast$-norm together with $\ast$-norm itself is a Banach space. 
\end{rem}
Our result from \cite{Pas15d}, inspired by \cite{Pas13} and the earlier result by P\'ales \cite{Pal91}, reads as follows
\begin{prop}
\label{prop:Pas15d_smooth}
Let $U$ be a closed, bounded interval and $f,\,g \in \Sm(U)$. Then 
\Eq{*}{
\rho(\sr{f},\sr{g})
\le \abs{U} \exp(\norma{\A{f}}_\ast) \cdot \left(\exp(\norma{\A{g}-\A{f}}_\ast)-1\right).
}
\end{prop}

Notice that the assumption that the set $U$ is closed could be skipped. Indeed, consider a sequence $U_1 \subset U_2 \subset \dots$ of closed intervals such that $\bigcup U_n=U$. Then every vector having entries in $U$ has also an entries in $U_n$ since a certain natural number $n$. Therefore, we may apply pertinent result to each of this set, and finally pass to the limit. 

Moreover the left hand side is symmetric with respect to $f$ and $g$, while the right one 
is not. One could clearly symmetrize this inequality using the $\min$ function. Nevertheless, this operation will be omitted to keep the notation compact. The same remark applies to the most of results among all the present paper.

%Paper \cite{Pas15d} contains an estimation which does not require smoothness of generated function using an operator introduced by P\'ales \cite{Pal91} which leads outside the scope of the present note. 

Very often we will use a global estimation of the operator $\A{}$ and, as it is handy, for $K > 0$ we denote
\Eq{*}{
\Sm_K(I):=\{f \in \Sm(I) \colon \norma{\A{f}}_{\infty} \le K\}.
}
By the virtue of Proposition~\ref{prop:basiccompare}, we may rewrite the definition of this family in the following way
\Eq{property:Sm_K_def}{
\Sm_K(I)=\{f \in \Sm(I) \colon \srE{-K}\le \sr{f} \le \srE{K}\}.
}
By \cite[Lemma~4.3--4.4]{Pas16a}, there exists a universal majorization of the difference between two means generated by a function from $\Sm_{K}$.
\begin{prop}
Let $U$ be a closed, bounded interval, $K>0$, and $f,\,g \in \Sm_K(U)$. Then 
\begin{itemize}
 \item $\rho(\sr{f},\sr{g}) \le \tfrac1K \ln \left(\tfrac12 \left( e^{K\abs{U}}+1\right)\right)$,
\item $\rho(\sr{f},\sr{g}) \le \tfrac{3+7e}6 K \abs{U}^2$.
\end{itemize}
\end{prop}

Two parts of this proposition are not comparable among each other; if an interval $\abs{U}$ is big, then the first inequality is better, for small $\abs{U}$ -- the second one.
Let me note that in the mentioned paper this proposition was stated for $K=1$, to skip this restriction we can apply the machinery described in \cite[section 4.1]{Pas16a}.

\section{\label{sec:main} Main result}

We have already presented a number of upper boundaries of the distance between two quasi-arithmetic means generated by functions from $\Sm(U)$. In this section we are going to present some lower boundary of this number. Throughout $U$ is a bounded interval, $K$ is a positive real.
 Our main tool is the following
\begin{prop}
\label{prop:main}
Let $f,g \in \Sm_{K}(U)$. If $\varepsilon:=\norma{\A{f}-\A{g}}_\ast$ then
\Eq{c5}{
\rho(\sr{f},\sr{g}) \ge 
\sup_{{c \in [0,1]} \atop{\delta \in (0 \tfrac\varepsilon{4K})}}
\min\Big( 
(1-c)(\tfrac \varepsilon{2K}-2\delta),\:
\frac{\delta(e^{\varepsilon/4}-1)(\exp\left((\tfrac \varepsilon2-2K\delta)\cdot c\right)-1)}{2 \cdot \exp(\norma{\A{f}}_\ast) (e^{K\abs{U}}-1)}
\Big)
}
 \end{prop}
Relevant proof is postponed until section \ref{sec:prooflem:1}. Having this in hand we are ready to prove the main theorem of the present note:
\begin{thm}
\label{thm:main}
Let $f,g \in \Sm_{K}(U)$. If $\varepsilon:=\norma{\A{f}-\A{g}}_\ast$ then
\Eq{*}{
\rho(\sr{f},\sr{g}) \ge \frac{\varepsilon(e^{\varepsilon/4}-1)(e^{\varepsilon/6}-1)}{16K \exp(\norma{\A{f}}_\ast)\left(e^{K\abs{U}}-1\right)}.
}
 \end{thm}

\begin{proof}
 
Upon putting $c=\tfrac23$ and $\delta=\tfrac \varepsilon{8K}$ in Proposition~\ref{prop:main}, we get
\Eq{*}{
\rho(\sr{f},\sr{g}) \ge \min\left(\tfrac \varepsilon{12K},\frac{\varepsilon(e^{\varepsilon/4}-1)(e^{\varepsilon/6}-1)}{16K \exp(\norma{\A{f}}_\ast)\left(e^{K\abs{U}}-1\right)}\right).
}

At the moment it is sufficient to prove that
\Eq{*}{
\frac{\varepsilon(e^{\varepsilon/4}-1)(e^{\varepsilon/6}-1)}{16K \exp(\norma{\A{f}}_\ast)\left(e^{K\abs{U}}-1\right)}\le \tfrac \varepsilon{12K}
}
or, equivalently,
\Eq{*}{
\frac{(e^{\varepsilon/4}-1)(e^{\varepsilon/6}-1)}{ \exp(\norma{\A{f}}_\ast)\left(e^{K\abs{U}}-1\right)}&\le \tfrac43.
}
Moreover, by the inequality $(e^x-1)(e^y-1) \le (e^{x+y}-1)$ valid for every positive reals, it suffices to prove that
\Eq{*}{
\frac{\exp(\tfrac5{12} \varepsilon)-1}{ \exp(\norma{\A{f}}_\ast)\left(e^{K\abs{U}}-1\right)}\le \tfrac43.
}

But in view of $f,\,g \in \Sm_K(U)$ we have $\varepsilon \in(0,2K\abs{U}\:]$. Therefore

\Eq{*}{
\frac{\exp(\tfrac5{12} \varepsilon)-1}{ \exp(\norma{\A{f}}_\ast)\left(e^{K\abs{U}}-1\right)}\le \frac{\exp(\tfrac5{6} K\abs{U})-1}{ \exp(\norma{\A{f}}_\ast)\left(e^{K\abs{U}}-1\right)}\le 1 < \tfrac43.
}

\end{proof}

 This theorem combined with Proposition~\ref{prop:Pas15d_smooth} immediately implies
\begin{cor}
\label{cor:top}
If $f,f_1,f_2,\dots \in \Sm_{K}(U)$ then
\Eq{*}{
\rho(\sr{f},\sr{f_n})\rightarrow 0 \iff \norma{\A{f}-\A{f_n}}_{\ast} \rightarrow 0.
}
 \end{cor}

 A problem of convergence of quasi-arithmetic means was already discussed among examples in \cite{Pas15d} and characterized for an arbitrary (not necessary differentiable) functions in \cite{Pal91}.

Theorem~\ref{thm:main} has an important disadvantage. The definition of $\Sm_K(U)$ implies $\varepsilon \le 2K\abs{U}$ (in fact this ineqality was already used in the proof of this theorem), whence the right hand side of the inequality stated as the main result is always smaller than $\tfrac18 \abs{U} e^{-K\abs{U}/2}$ (this technical estimation is omitted).
To avoid this drawback we will use the simple fact that the distance between means are the lowest upper boundary of the distance between means taking for every vector and weights. Whence, if we restrict the set of admissible vectors or weight then the distance will not increase. In particular we can take only a vectors from some subinterval $V \subset U$. More precisely, for every continuous, monotone functions $f,g \colon U \rightarrow \R$ the following inequality holds:
\Eq{*}{
\rho(\sr{f},\sr{g}) \ge \rho(\sr{f \vert_V}, \sr{g \vert_V}).
}
While we are taking a subinterval $V$ of $U$ we need to control the distance $\norma{\A{f}-\A{g}}_{\ast,V}$. Luckily we have the following
\begin{lem}[Partitioning lemma]
\label{lem:split}
Let $U$ be an interval, $u \in \mathcal{C}(U)$ and $n \in \N$. There exists a subinterval $V \subset U$ such that $\abs{V} = \tfrac1n \abs{U}$ and 
 $\norma{u}_{\ast,V} \ge \tfrac1n \norma{u}_{\ast,U}$.
\end{lem}
\begin{proof}
Take a partition $(V_i)_{i=1}^n$ of $U$ such that $\abs{V_i}=\tfrac1n \abs{U}$ for every $i \in \{1,\dots,n\}$.
Then, by the triangle inequality,
\Eq{*}
{
\norma{u}_{\ast,U}=\big\|\sum_{i=1}^n u\vert_{V_i}\big\|_{\ast,U}\le \sum_{i=1}^n \norma{u\vert_{V_i}}_{\ast,U}= \sum_{i=1}^n \norma{u}_{\ast,V_i}.
}
In particular, $\norma{u}_{\ast,V_j} \ge \tfrac1n \norma{u}_{\ast,U}$ for some $j\in \{1,\dots,n\}$.
\end{proof}

At the moment we will divide our consideration onto two cases. First possibility is that factor $e^{K\abs{U}}-1$ appearing in the denominator of the inequality in Theorem~\ref{thm:main} is majorized by a constant number (first case). Otherwise, having Lemma~\ref{lem:split} in hand, we split the set $U$ obtaining a subinterval $V \subset U$ of length comparable to $\tfrac1K$. In this setting $e^{K\abs{V}}$ becomes a constant number. This idea is quite simple, but there appear a number of artificial constants both in its wording and proof.

\begin{cor}
 \label{cor:estim}
Let $f,g \in \Sm_{K}(U)$.
The mapping $ C \mapsto \frac{C^3}{3072 \cdot e^C(e^C-1)}$ achieve its maximal value $y_0 \approx \constantyzero$ for $C=C_0 \approx \constantCzero$. Moreover, let $y_1=\frac{1}{384 \cdot \exp(C_0/2)(e^{C_0/2}-1)} \approx \constantyone$.
\begin{itemize}
 \item[\textrm{(i)}] If $K \abs{U} \le \frac{C_0}2$ then
\Eq{*}{
\rho(\sr{f},\sr{g})\ge  y_1 \frac{\norma{\A{f}-\A{g}}_\ast^3}{K}.
}
 \item[\textrm{(ii)}] If $K \abs{U} \ge \frac{C_0}2$ then
\Eq{*}{
\rho(\sr{f},\sr{g})\ge y_0 \frac{\norma{\A{f}-\A{g}}_\ast^3}{\abs{U}^3 \cdot K^4}.
}
 \end{itemize}
\end{cor}

Before we begin the proof let me notice that in view of this corollary $\rho(\srE{15} \vert_{(0,1)},\srE{20} \vert_{(0,1)}) \ge 5.71442 \cdot 10^{-8}$ while in fact $\rho(\srE{15} \vert_{(0,1)},\srE{20} \vert_{(0,1)}) \approx 0.212$. Significant disproportion between this boundary and the real distance is caused mainly by the fact that there are only three parameters appearing in the right hand side of Corollary~\ref{cor:estim}, while the nature quasi-arithmetic means is much more complicated.

\begin{proof}
Let us denote, as usually, $\varepsilon:=\norma{\A{f}-\A{g}}_{\ast}$. We are going to prove each part of this corollary separately. 

\textbf{Part \textrm{(i)}.} By $\norma{\A{f}}_{\ast} \le K\abs{U} \le \tfrac{C_0}2$ and a common inequality $e^x-1 \ge x$ one obtains
\begin{align*}
\rho(\sr{f},\sr{g})&\ge \frac{\varepsilon(e^{\varepsilon/4}-1)(e^{\varepsilon/6}-1)}{16K \exp(\norma{\A{f}}_\ast)\left(e^{K\abs{U}}-1\right)} 
\ge \frac{\varepsilon^3}{384K \exp(K\abs{U})\left(e^{K\abs{U}}-1\right)} \\
&\ge \frac{1}{384K \exp(C_0/2)\left(e^{C_0/2}-1\right)} \varepsilon^3 
= y_1 \frac{\varepsilon^3}K.
\end{align*}

\textbf{Part \textrm{(ii)}.} 
Define 
\Eq{*}{
n_0:=\left\lceil \frac{K\abs{U}}{C_0} \right\rceil.
}
We know that $K\abs{U}/C_0 \ge \tfrac12$, and therefore
\Eq{Estn0}{
n_0 \in \left[ \frac{K\abs{U}}{C_0}, \frac{2K\abs{U}}{C_0} \right].
}
By Lemma~\ref{lem:split} consider an interval $V \subset U$ such that $\abs{V} = \tfrac 1{n_0} \abs{U}$ and 
\Eq{*}{
\varepsilon':=\norma{\A{f}-\A{g}}_{\ast, V} \ge \tfrac1{n_0} \norma{\A{f}-\A{g}}_{\ast, U}=\varepsilon/n_0.
}

By \eqref{Estn0},
\Eq{Un0}
{ K \abs{V} = \tfrac1{n_0} K \abs{U}\in [\tfrac{C_0}{2}, C_0].}
Moreover
\Eq{U5}{
\varepsilon' \ge \frac{ \varepsilon}{n_0}
\ge\frac{\varepsilon \cdot C_0}{2 K \abs{U}}. 
}
On the other hand, by $f \in \Sm_K(U)$, we can majorize 
$\norma{\A{f}}_{\ast, V} \le K \abs{V} \le C_0$.
Finally, by \eqref{Un0}, \eqref{U5}, and the classical inequality $e^x-1\ge x$ we obtain
\begin{align}
\rho(\sr{f},\sr{g})&\ge \frac{\varepsilon'(e^{\varepsilon'/4}-1)(e^{\varepsilon'/6}-1)}{16K \exp(\norma{\A{f}\vert_{V}}_\ast)\left(e^{K\abs{V}}-1\right)} \nolabel{U7} 
%%%%%%%%%%%%%
\ge \frac{(\varepsilon')^3}{384K \exp(K \abs{V})\left(e^{K\abs{V}}-1\right)} \nolabel{U7} \\
%%%%%%%%%%%%%
&\ge \frac{C_0^3}{384 \cdot 2^3 \cdot e^{C_0}(e^{C_0}-1)} \cdot \frac{\varepsilon^3}{\abs{U}^3K^4}  
=y_0 \cdot \frac{\varepsilon^3}{\abs{U}^3K^4} \nonumber
\end{align}

\end{proof}

\section{\label{sec:box} Box Distance}

In the previous section distance between means generated by $f,g \in \Sm_{K}(U)$ ($U$ - an interval) was expressed in terms of $\norma{\A{f}-\A{g}}_\ast$ -- the main theorem stated that $\rho(\sr{f},\sr{g})$ may be estimated from below by some term involving this value, length of the interval, and the number $K$. 

In this section we will define the distance between generators in the other way. More 
precisely we say that $f, g \in \Sm(U)$ ($U$ -- an interval) are \emph{$(\phi,K,\delta)$-separated} 
if there exist a closed interval $V \subset U$, $\abs{V} = \phi$ such that for all $x \in V$ the following inequalities are satisfied: $\abs{\A f(x)}\le K$, $\abs{\A g(x)}\le K$, and $\abs{\A f(x)-\A g(x)}\ge\delta$. 

Let me note that $(\phi,K,\delta)$-separation does not imply that functions belong to $\Sm_K(U)$, because the majorization of the Arrow-Pratt index is only on some subinterval (denoted above by $V$). Notice that both $f\vert_{V},\:g \vert_{V} \in \Sm_K(V)$ -- this makes using the letter $K$ absolutely natural in this context. 
As we will see, results from the previous section are useless here. We are going to prove the following statement:

\begin{thm}
\label{thm:box}
If $f,g \in \Sm(U)$ are $(\phi,K,\delta)$-separated then
\begin{align*}
\rho(\sr{f},\sr{g})&\ge \tfrac1K \ln(1+K \alpha),
\end{align*}
where
\begin{align}
\alpha&:= \frac{e^{-K\phi/2}-1}{K}-\frac{e^{(\delta-K)\phi/2}-1}{K-\delta}.\nolabel{2.162}
\end{align}
 \end{thm}

 \begin{proof}

Let us take an interval $V$ from the definition of $(\phi,K,\delta)$-separation.
Denote $x:=\inf V$, $z:=\sup V$, and $y:=\tfrac{x+z}2$; in the other words $x,\,y,\,z \in V$ and $y-x=z-y=\phi/2$. In view of Remark~\ref{rem:eq_qa} we may assume without loss of generality that
\Eq{*}
{
f(y)=g(y)=0 \text{ and } f'(y)=g'(y)=1.
}
By the definition of separation and Proposition~\ref{prop:basiccompare}, we may also assume
\Eq{*}{
\A{g}(x)=\A{f}(x)+\delta, \quad x \in V.
}
Then, for every $v \in V$,
\begin{align}
g(v)&=\int_y^v e^{\int_y^t \A{g}(u)du}dt
=\int_y^v e^{\int_y^t \A{f}(u)du} \cdot e^{\delta \cdot (t-y)} dt \nolabel{2.5}\\  
&= \int_y^v e^{\delta(t-y)} f'(t) dt 
= f(v)+ \int_y^v \left(e^{\delta(t-y)}-1\right) f'(t) dt \label{2.7}.
\end{align}
Moreover, by $\A{f}(x) \in [-K,K-\delta]$ for $x \in U$, we immediately obtain two inequalities
\begin{align}
f'(t) \ge e^{(K-\delta) (t-y)}\quad \text{ for } t<y, \label{2.8}\\
f'(t) \ge e^{-K (t-y)}\quad \text{ for }t>y. \label{2.9}
\end{align}

This inequalities allow us to estimate [from below] values of $g(x)$ and $g(z)$. Indeed, \eqref{2.7} and \eqref{2.8} follow
\begin{align}
g(x)
&\ge f(x)+\int_x^y \left(1-e^{\delta(t-y)}\right) e^{(K-\delta)(t-y)} dt \nolabel{2.14}\\
&=f(x)+\frac{1-e^{(K-\delta)(x-y)}}{K-\delta}+\frac{e^{K(x-y)}-1}{K}
=f(x)+\alpha.\label{2.19}
\end{align}

Analogously, \eqref{2.7} and \eqref{2.9} imply
\Eq{2.18}{
g(z) &\ge f(z)+\alpha.
}

At the moment let us take $\theta \in (0,1)$ satisfying $\sr{f}_\theta(z,x)=y$. Then, by the definition of quasi-arithmetic mean, we get
\Eq{*}{
0=f(y)=\theta f(z)+(1-\theta) f(x). \nolabel{2.21}
}
Now we are going to bound the mean $\sr{g}$ for the same arguments. First we prove 
\Eq{2.27}{
g(\sr{g}_\theta(z,x)) \ge \alpha.
}
Indeed, by \eqref{2.19} and \eqref{2.18},
\begin{align*}
g(\sr{g}_\theta(z,x))&=\theta g(z)+(1-\theta) g(x)\\
&\ge \theta f(z)+\theta\alpha+(1-\theta)f(x)+(1-\theta)\alpha\\
&=f(y)+\alpha =\alpha.
\end{align*}
On the other hand $g(y)=0$ and $\abs{\A{g}(x)}\le K$ for $x \in V$. Thus we simply have
\Eq{*}{
g(u)\le \int_y^u e^{\int_y^x K dt}dx=\tfrac1K(e^{K(u-y)}-1),\quad u \in[y,z]. \nonumber
}
Combining this inequality with \eqref{2.27}, we obtain
\Eq{*}{
\alpha\le g(\sr{g}_\theta(z,x)) &\le \tfrac 1K \left( e^{K\cdot(\sr{g}_\theta(z,x)-y)}-1\right). }
Recall that $y=\sr{f}_\theta(z,x)$. Then one can rewrite the inequality above in the alternative form 
\Eq{*}{
 \sr{g}_\theta(z,x)- \sr{f}_\theta(z,x) &\ge  \tfrac1K \ln \left(1+ K\alpha \right).
 }
\end{proof}

Number $\alpha$ appearing in Theorem~\ref{thm:box} is rather a complicated one. Nevertheless, we can observe that $\alpha$ is a difference of one simple function evaluated in two different points. We can use this to simplify the right hand side. 
\begin{cor}
\label{cor:box}
 If $f,g \in \Sm(U)$ are $(\phi,K,\delta)$-separated then
\Eq{*}{
\rho(\sr{f},\sr{g})\ge\tfrac1K \ln \left(1+\tfrac\delta K \cdot \Theta (\tfrac{K\phi}2)\right),
}
where $\Theta(x):=1-e^{-x}-xe^{-x}$.
\end{cor}
\begin{proof}
Notice that the function 
\Eq{*}{
\omega(x):=\frac{e^{-x}-1}x
}
is increasing and concave, and therefore $\omega'$ is decreasing. By mean value theorem we get
\Eq{*}{
\omega(\tfrac{K\phi}2)-\omega(\tfrac{(K-\delta)\phi}2) \ge \tfrac{\delta\phi}2 \omega'(\tfrac{K\phi}2).
}
Applying this, Theorem~\ref{thm:box}, and the algebraic identity $x^2\omega'(x)=\Theta(x)$ we have 
\begin{align}
 \rho(\sr{f},\sr{g})
 &\ge\tfrac1K \ln \left(1+\tfrac \phi2 \cdot \left(\omega(\tfrac{K\phi}2)-\omega(\tfrac{(K-\delta)\phi}2)\right)\right) \nonumber \\
 &\ge\tfrac1K \ln \left(1+\tfrac \phi2 \cdot \tfrac{K\delta\phi}2\omega'(\tfrac{K\phi}2) \right) \nonumber \\
 &=\tfrac1K \ln \left(1+\tfrac \delta K \cdot \Theta(\tfrac{K\phi}2) \right) \nonumber 
\end{align}
\end{proof}

At the end of this section let us present an applications of all results on a simple example. Let me notice that the order of this numbers may vary depending on means. We take all additional parameters appearing in each result to obtain best possible boundaries.
\begin{align*}
\text{real value} &\qquad \rho(\srE{15} \vert_{(0,1)},\srE{20} \vert_{(0,1)}) \approx 0.212; \\
\text{Theorem~\ref{thm:main}} &\qquad \rho(\srE{15} \vert_{(0,1)},\srE{20} \vert_{(0,1)}) \ge 
3.19184 \cdot 10^{-17}; \\
\text{Corollary~\ref{cor:estim}} &\qquad \rho(\srE{15} \vert_{(0,1)},\srE{20} \vert_{(0,1)}) \ge 5.71442 \cdot 10^{-8}; \\
\text{Theorem~\ref{thm:box}} &\qquad \rho(\srE{15} \vert_{(0,1)},\srE{20} \vert_{(0,1)}) \ge 0.0143; \\
\text{Corollary~\ref{cor:box}} &\qquad \rho(\srE{15} \vert_{(0,1)},\srE{20} \vert_{(0,1)}) \ge 0.011.
\end{align*}

\section{\label{sec:prooflem:1} Proof of Proposition~\ref{prop:main}}
We will prove that \eqref{c5} holds for every $\varepsilon \in (0, \norma{\A{f}-\A{g}}_{\ast})$. It implies that  \eqref{c5} holds for $\varepsilon =\norma{\A{f}-\A{g}}_{\ast}$ too. 
Fix $\varepsilon \in (0, \norma{\A{f}-\A{g}}_{\ast})$, $c \in [0,1]$, and $\delta \in (0, \tfrac{\varepsilon}{4K})$. We will prove that 

\Eq{c5_1}{
\rho(\sr{f},\sr{g}) \ge 
\min\Big( 
(1-c)(\tfrac \varepsilon{2K}-2\delta),\:
\frac{\delta(e^{\varepsilon/4}-1)(\exp\left((\tfrac \varepsilon2-2K\delta)\cdot c\right)-1)}{2 \cdot \exp(\norma{\A{f}}_\ast) (e^{K\abs{U}}-1)}
\Big),}

which implies the inequality \eq{c5} to be proved.

In what follows this proof will be split into two cases. In each of them we will prove that $\rho(\sr{f},\sr{g})$ can be bounded from below be one of terms appearing on the right hand side of \eqref{c5_1}.
 As far as we are not able to predict which case is valid, we have to use a $\min$ function in the final result.
Before we begin the proper part let me introduce two technical, however important, constants:
\begin{align}
\theta&:=\frac{e^{\varepsilon/4}-1}{e^{K\abs{U}}-1},\label{11_5}\\
\alpha&:=\exp \left(2K\delta-\tfrac \varepsilon2 \right). \label{alpha_def}
\end{align}

At the moment we will briefly describe the idea of this proof.
We define:
\Eq{*}
{
F:=\sr{f}_\theta(x,z), \qquad
G:=\sr{g}_\theta(x,z ), \qquad
\tilde F:=\sr{f}_\theta(x+\delta,z ), \qquad
\tilde G:=\sr{g}_\theta(x+\delta,z ). 
}
Equalities $F=G$ or $\tilde F=\tilde G$ cannot be excluded, therefore to provide certain lower bound we need to use more sophisticated way. 
Namely, we will not just estimate the difference between two means, but $\max(\abs{F-G},\abs{\tilde F-\tilde G})$. Then we will use the trivial inequality
\Eq{16_5}{
\rho(\sr{f},\sr{g}) \ge \max(\abs{F-G},\abs{\tilde F-\tilde G}).
}

 By the definition of $\norma{\cdot}_{\ast}$, there exist $x,\,z \in U$, $x<z$ such that, 
\Eq{*}{
\abs{\int_x^z \A{f}(t)-\A{g}(t)dt}=\varepsilon.
}
By Remark~\ref{rem:eq_qa}, let $f(x)=g(x)=0$ and $f'(x)=g'(x)=1$. Assume without loss of generality, switching $f$ and $g$ if necessary, that
\Eq{Emain:1}{
\int_x^z \A{f}(t)-\A{g}(t)dt=\varepsilon.
}
Then either  
\Eq{*}{
\int_x^{(x+z)/2} \A{f}(t)-\A{g}(t)dt\le \tfrac\varepsilon2 \text{ or } \int_{(x+z)/2}^z \A{f}(t)-\A{g}(t)dt\le \tfrac\varepsilon2.
}
This cases are analogous (the mapping $f(t) \mapsto -f(\frac{x+y}2-t)$ is a natural transitions between them; cf. e.g. \cite{Pas13,Pas15d,Pas15a}). Therefore, from now on we may assume that the first inequality holds. By  \eqref{Emain:1} there exists $y \in [\tfrac{x+z}2,z)$ such that
\begin{align}
\int_x^y \A{f}(t)-\A{g}(t) dt&=\tfrac\varepsilon2,\quad \text{ and } \label{7}\\
\int_x^u \A{f}(t)-\A{g}(t) dt&\ge \tfrac\varepsilon2\quad \text{ for all }u \in (y,z). \label{8}
\end{align}

Equality \eqref{7} can be expressed equivalently as 
\Eq{7_1}{
\int_y^z \A{f}(t)-\A{g}(t) dt=\tfrac{\varepsilon}2.
}
Moreover, by \eqref{8} and the identity 
\Eq{ID_AP}{
f'(u)=\exp(\int_x^u \A f(t) - \A g(t) dt) g'(u), \quad u \in U,
}
we have
\Eq{10}{
f'(u)&\ge e^{\varepsilon/2} g'(u),\quad u\in [y,z].
}

On the other hand, by $f,g \in \Sm_K(U)$, we get $\abs{\A{f}(t)-\A{g}(t)} \le 2K$ for every $t \in U$. Therefore \eqref{7_1} implies $z-y \ge\tfrac{\varepsilon}{4K}$. The same estimation applied to \eqref{ID_AP} follows
\begin{align}
f'(u)&\le e^{2K \cdot \abs{x-u}} g'(u), \quad u \in U. \label{11}
\end{align}

Definition of $\Sm_{K}$ and $\theta$ expressed in \eqref{property:Sm_K_def}, \eqref{11_5}, respectively, follow
\begin{align}
\min(F,G) \ge \srE{-K}\Big((x,z),(\theta,1-\theta)\Big)= z-\tfrac{\varepsilon}{4K} \ge y.
\end{align}
Whence $F \ge y \ge \tfrac{x+z}2$ and $G \ge y \ge \tfrac{x+z}2$.
Furthermore we have a simple equalities
\begin{align}
f(\tilde F)-f(F)&=\theta (f(x+\delta)-f(x)), \label{16_75} \\
g(\tilde G)-g(G)&=\theta (g(x+\delta)-g(x)). \nonumber
\end{align}
This equalities combined with  \eqref{11} and \eqref{10} implies
\begin{align*}
 f(\tilde F)-f(F)&=\theta(f(x+\delta)-f(x))
\le \theta e^{2K \delta} (g(x+\delta)-g(x)) 
= e^{2K \delta} \cdot \left(  g(\tilde G)-g(G) \right)\\
&\le \frac{e^{2K\delta}}{e^{\varepsilon/2}} \left(  f(\tilde G)-f(G) \right)
= \alpha \left(  f(\tilde G)-f(G) \right).
\end{align*}
In fact the inequality above is crucial. We know that $\alpha<1$ and $f(\tilde F)-f(F)$ is positive. Therefore this inequality alone implies that equalities $F=G$ and $\tilde F=\tilde G$ cannot be simultaneously satisfied. This simple idea allows us to estimate lower boundary of the difference between this value. To do this, let us express the inequality above in the integral form
\Eq{*}{
(\tilde F-F) \cdot \int_0^1 f'(F+(\tilde F-F) \theta) d\theta &< \alpha \cdot (\tilde G-G) \cdot  \int_0^1 f'(G+(\tilde G-G) \theta) d\theta.
}
By the definition $\tilde F > F$, $\tilde G >G$, and $f'(x)>0$ for every $x \in U$. Thus either
\begin{enumerate}[(i)]
 \item $\int_0^1 f'(F+(\tilde F-F) \theta) d\theta  \le \alpha^{1-c} \cdot \int_0^1 f'(G+(\tilde G-G) \theta) d\theta $, or
 \item $\tilde F-F \le \alpha^c \cdot (\tilde G-G)$.
 \end{enumerate}
This naturally splits our proof onto two cases depending on which of inequalities hold. It could happen that both of them hold, but it does not affect to the proof.

\subsection{Case (i)}
By mean value theorem, there exists $\theta_0 \in (0,1)$ such that
\begin{align}
f'(F+(\tilde F-F) \theta_0) &\le \alpha^{1-c} \cdot f'(G+(\tilde G-G) \theta_0), \nonumber \\
\alpha^{c-1} &\le \frac{f'(G+(\tilde G-G) \theta_0)}{f'(F+(\tilde F-F) \theta_0)}. \nonumber
\end{align}
On the other hand by $f \in \Sm_{K}(U)$ we get
\begin{align*}
\abs{\frac{d}{d t}(\ln f'(t))}&\le K,\\
\abs{\ln f'(G+(\tilde G-G) \theta_0)-\ln f'(F+(\tilde F-F) \theta_0)}&\le K\abs{(G+(\tilde G-G) \theta_0)-(F+(\tilde F-F) \theta_0)} \\
\frac{f'(G+(\tilde G-G) \theta_0)}{f'(F+(\tilde F-F) \theta_0)}&\le \exp \left(K \abs{(1-\theta_0)(G-F)+\theta_0(\tilde G-\tilde F)}\right).
\end{align*}
But $\theta_0 \in(0,1)$, so we simply obtain 
\Eq{*}{
\alpha^{c-1} \le \frac{f'(G+(\tilde G-G) \theta_0)}{f'(F+(\tilde F-F) \theta_0)}&\le \exp \left(K \cdot \rho(\sr{f},\sr{g})\right).
}
Finally in this case we have the inequality
\Eq{*}{
\rho(\sr{f},\sr{g}) \ge \frac{(c-1)\ln(\alpha)}{K} 
= \tfrac{(c-1)(2K\delta-\tfrac \varepsilon2)}{K} 
=(1-c)(\tfrac \varepsilon{2K}-2\delta).
}

\subsection{Case (ii)}
%We have
%\begin{align*}
%\tilde G-G &\ge \alpha^{-c}  (\tilde F-F),\\
%\tilde G-\tilde F-G+F &\ge (\alpha^{-c}-1)  (\tilde F-F).\\
%\end{align*}
Using the elementary inequality $\max(\abs{p},\abs{q}) \ge \tfrac{p-q}2$ we get
\Eq{a3}{
\max\left(\abs{\tilde G-\tilde F},\abs{G-F}\right)&\ge \tfrac12 (\tilde G-\tilde F-G+F) \ge \tfrac12 (\alpha^{-c}-1) (\tilde F-F).}

Notice that this step is in fact main reason of the huge disproportion between our estimation and the optimal one. Assume for example that $G-F \approx \rho(\sr{f},\sr{g})$. As $\tilde F$ and $\tilde G$ are close to $F$ and $G$ respectively, we obtain that $\tilde G-\tilde F \approx \rho(\sr{f},\sr{g})$ too. However in this case we use a boundary
$\rho(\sr{f},\sr{g}) \ge \tfrac12 ((\tilde G-\tilde F)-(G-F))$,
which is far from $\rho(\sr{f},\sr{g})$.

By mean value theorem, there exists $\mu \in(F,\tilde F)$ and $\nu \in(x,x+\delta)$ such that
\Eq{*}{
f(\tilde F)-f(F)&=f'(\mu)(\tilde F-F),\\
f(x+\delta)-f(x)&=f'(\nu) \cdot \delta.
}
At the moment \eqref{16_75} can be rewritten as
\Eq{*}{
\tilde F-F=\theta \delta \frac{f'(\nu)}{f'(\mu)}, \nolabel{a6}
}

At the moment we are going to use the inequality from \cite{Pas15d}:
\Eq{*}{
\frac{f'(\nu)}{f'(\mu)}
=\exp(\int_\mu^\nu \A{f}(x)dx) 
\ge \exp(-\norma{\A{f}}_{\ast}).
}
Thus we immediately obtain
\Eq{a7}{
\tilde F-F \ge \frac{\theta \delta}{\exp(\norma{\A{f}}_\ast)}.
}

Finally, combining \eqref{a3}, \eqref{a7}, \eqref{alpha_def}, and \eqref{11_5} we obtain
\begin{align}
\rho(\sr{f},\sr{g}) &\ge \max\left(\abs{\tilde G-\tilde F},\abs{G-F}\right) \nonumber \\
&\ge 
\tfrac12 (\alpha^{-c}-1) (\tilde F-F) \nolabel{}\\
&\ge\frac{\theta \delta}{2 \cdot \exp(\norma{\A{f}}_\ast)} (\alpha^{-c}-1) \nolabel{a8}\\
%&\ge \frac{\theta \delta}{2 \cdot \exp(\norma{\A{f}}_\ast)} (\exp\left((\tfrac \varepsilon2-2K\delta)\cdot c\right)-1) \nolabel{a9}\\
&= \frac{\delta}{2 \cdot \exp(\norma{\A{f}}_\ast)} \cdot \frac{e^{\varepsilon/4}-1}{e^{K\abs{U}}-1} \cdot (\exp\left((\tfrac \varepsilon2-2K\delta)\cdot c\right)-1), \label{a10}
\end{align}
which is the second term appearing in the right hand side of \eqref{c5_1}.

%\def\cprime{$'$} \def\R{\mathbb R} \def\Z{\mathbb Z} \def\Q{\mathbb Q}
%  \def\C{\mathbb C}

%\bibliography{publ,funcequ}
%\bibliographystyle{plain}

\end{document}